\documentclass[12pt,oneside]{amsart}

\usepackage{amsthm,amsmath,amssymb}
\usepackage{mathrsfs}
\usepackage{enumerate}
\usepackage{amsfonts}
\usepackage{verbatim}
\usepackage{amsbsy}
\usepackage{amsmath}
\usepackage{amssymb}
\usepackage{MnSymbol}
\usepackage{mathrsfs}

\newtheorem{theorem}{Theorem}[section]
\newtheorem{lemma}[theorem]{Lemma}
\newtheorem{proposition}[theorem]{Proposition}

\newtheorem{corollary}[theorem]{Corollary} 
\theoremstyle{definition}
\newtheorem{definition}[theorem]{Definition}

\newtheorem{idea}[theorem]{Idea}
\theoremstyle{remark}
\newtheorem{remark}[theorem]{Remark}
\newtheorem{example}[theorem]{Example}
\newtheorem{case}{Case}
\newtheorem{subcase}{Case}
\numberwithin{subcase}{case}

\DeclareMathOperator{\supp}{supp}

\DeclareMathOperator{\core}{core}
\DeclareMathOperator{\pos}{pos}
\DeclareMathOperator{\Sym}{Sym}

\usepackage{calc}
\usepackage{tikz}
\usetikzlibrary{decorations.markings}
\tikzstyle{vertex}=[circle, draw, inner sep=0pt, minimum size=6pt]
\newcommand{\vertex}{\node[vertex]}


\begin{document}

\title{Parikh Word Representability of Bipartite Permutation Graphs}
\author{Wen Chean Teh, Zhen Chuan Ng, Muhammad Javaid, Zi Jing Chern}
\address{School of Mathematical Sciences\\
Universiti Sains Malaysia\\
11800 USM, Malaysia}
\email[Corresponding author]{dasmenteh@usm.my}
\address{School of Mathematical Sciences\\
	Universiti Sains Malaysia\\
	11800 USM, Malaysia}
\email{zhenchuanng@usm.my}
\address{Department of Mathematics, School of Science, University of Management and Technology,  Lahore, Pakistan}
\email{javaidmath@gmail.com}
\address{School of Mathematical Sciences\\
	Universiti Sains Malaysia\\
	11800 USM, Malaysia}
\email{jing\_12364@hotmail.com}

\maketitle

\begin{abstract}
The class of Parikh word representable graphs were recently introduced. In this work,
we further develop its general theory beyond the binary alphabet. Our main result shows that this class is equivalent to the class of bipartite permutation graphs. Furthermore, we study certain graph theoretic properties of these graphs in terms of the arity of the representing word.
\end{abstract}

\section{Introduction}

The theory of intersection graphs is well-studied and many important graph families are special classes of intersection graphs (see \cite{mckee1999topics}).
Intersection graphs of sets of line segments in the plane are particularly interesting because it was conjectured by Sheinerman and proved by Chalopin and Gon{\c{c}}alves \cite{chalopin2009every} that every planar graph is such an intersection graph.
The class of permutation graphs is equivalent to a very special subclass of this class, where the endpoints of the line segments lie on two parallel lines.
In 2010 Spinrad~et al.~\cite{spinrad1987bipartite} gave the first characterization of bipartite permutation graphs and since then this class has been proven to be equivalent to a few other natural classes of graphs, for example, the proper interval bigraphs \cite{hell2004interval}.


The theory of word representable graphs is a very young but well-established research area, which relates graph theory to combinatorics on words. 
An excellent survey of the state of the art would be \cite{kitaev2015words}.
Parikh word representable graphs, independent from word representable graphs, were introduced recently \cite{bera2016structural}
as a new approach to study words using graphs and vice versa.
It was proven that every Parikh binary word representable graph is a bipartite permutation graph. In this work, we generalize this result to arbitrary ordered alphabets and show that up to isomorphism, a graph is Parikh word representable if and only if it is a bipartite permutation graph. Therefore, we obtain a new characterization of the class of bipartite permutation graphs.

Parikh matrices \cite{mateescu2001sharpening} was introduced in 2001 as an extension of the classical Parikh vectors \cite{parikh1966context}. The injectivity problem, which asks when are two words having the same Parikh matrices, is still open even for the ternary alphabet for almost two decades
and thus received our attention lately (for example, see \cite{poovanandran2018elementary, teh2016conjecture,  teh2018strongly, teh2018order}).  The definition of Parikh word representable graph is originally motivated by Parikh matrices. In fact, it is closely related to the definition of core of a word introduced in \cite{teh2015coreb} to study the injectivity problem (see Remark~\ref{0910a}).

The remainder of this paper is structured as follows. Section 2 provides the basic terminology and preliminaries. Section 3 reviews the basics of Parikh word representable graphs and also presents some new observations. In Section~4, we present our main result that says that the Parikh word representable graphs are exactly the bipartite permutation graphs. Hence, this induces a hierarchy for bipartite permutation graphs based on the arity of representing words. Subsequently, the upper bound of diameters of Parikh graphs  representable by $n$-ary words is studied in Section~5. Before, our conclusion, in Section 6, we restrict our focus to the binary and ternary cases and extend some results from \cite{bera2016structural} analogously to the ternary alphabet. In particular, the necessary and sufficient condition for a Parikh ternary word representable graph to have a Hamiltonian cycle is obtained. 

\section{Preliminaries}

To our best knowledge, word or graph theoretic terminology used in this work but not detailed here are the standard ones.

Suppose $\Sigma$ is an alphabet. The set of words over $\Sigma$ is denoted by $\Sigma^*$. The empty word is denoted by $\lambda$. 
Let $\Sigma^+$ denote the set $\Sigma^*\backslash\{\lambda\}$.
An \emph{ordered alphabet} is an alphabet with an ordering on it, for example, $\{a_1<a_2< \dotsb<a_s\}$.
An ordered alphabet and its underlying alphabet can both be denoted by $\Sigma$. For $1\leq i\leq j \leq s$, let $a_{i,j}$ denote the word $a_ia_{i+1}\dotsm a_j$.
Suppose $\Gamma\subseteq \Sigma$. The projective morphism $\pi_{\Gamma}\colon \Sigma^*\rightarrow \Gamma^*$ is defined by
$$\pi_{\Gamma}(a)=\begin{cases}
a, & \text{if } a\in \Gamma\\
\lambda, & \text{otherwise}
\end{cases}.$$
We may write $\pi_{a,b}$ for $\pi_{\{a,b\}}$.

Suppose $w\in \Sigma^*$. The length of $w$ is denoted by $\vert w\vert$.
A word $w'$ is a \emph{subword} of $w\in \Sigma^*$ if and only if there exist $x_1,x_2,\dotsc, x_n$, $y_0, y_1, \dotsc,y_n\in \Sigma^*$, possibly empty, such that
	$$w'=x_1x_2\dotsm x_n \text{ and } w=y_0x_1y_1\dotsm y_{n-1}x_ny_n.$$
We say that $ u $ is a \emph{factor} of $ w $ if and only if there exist $ x, y \in \Sigma^* $ such that $ w = xuy $. If $ x $ (respectively $ y $) is the empty word, then $ u $ is called a \emph{prefix} (respectively \emph{suffix}) of $ w $. The number of occurrences of a word $u$ as a subword of $w$ is denoted by $\vert w\vert_u$. 
Two occurrences of $u$ are considered different if and only if they differ by at least one position of some letter. By convention, $\vert w\vert_{\lambda}=1$ for all $w\in \Sigma^*$. The \emph{support} of $w$, denoted $\supp(w)$, is the set $\{ a\in \Sigma\mid \vert w\vert_a \neq 0 \}$.

\begin{definition}
	Suppose $\Sigma$ is an alphabet, $w\in \Sigma^*$ and $a\in \supp(w)$. For every $1\leq k\leq \vert w\vert_a$, let $\pos_a(w,k)$ denote the position of the $k$-th character in $w$ that is $a$.
\end{definition}

\begin{example}
	$\pos_b(abbaba,2)=3$ while $\pos_{a}(caabcaba,3)=6$.
\end{example}

\begin{definition}\cite{teh2015coreb}
	Suppose $\Sigma$ is any alphabet and $v,w\in\Sigma^*$. The {\it v-core} of $w$, denoted by $\core_v(w)$, is the unique subword $w'$ of $w$ such that $w'$ is the subword of the shortest length which satisfies $|w'|_v=|w|_v$.
\end{definition}
In other words, $\core_v(w)$, is the subword of $w$ consisting of letters in $w$ that contribute to the value of $\vert w\vert_v$. 

\begin{example}
	Consider the word $w=bacbbabcccbac$. Then $\core_b(w)=bbbbb$, $\core_{ab}(w)=abbabb$, $\core_{bc}(w)=bcbbbcccbc$, $\core_{abc}(w)=abbabcccbc$,  $\core_{cab}(w)=cabb$, and $\core_{cca}=cccca$.
\end{example}

We will be working only with undirected simple graphs with no loops and multiple edges. A \emph{graph} $G=(V,E)$ is a pair consisting of a set $V$ of \emph{vertices}, denoted by $V(G)$, and a set $E$ of (undirected) \emph{edges}, denoted by $E(G)$.  
Suppose $x,y\in V$ are vertices of $G$.
If $x$ and $y$ are adjacent, we let $(x,y)$ denote the edge connecting   
$x$ and $y$ and identify it with  $(y,x)$. The \emph{open neighborhood} of $x$ is denoted by $N(x)$. The \emph{degree} of $x$, denoted $\deg(x)$, is the number of vertices \emph{adjacent} to $x$.
The \emph{distance} between $x$ and $y$
is denoted by $d(x,y)$. The \emph{diameter} of a graph is the greatest distance between any pair of vertices.
The \emph{induced subgraph} of $G$ induced by the subset $V'\subseteq V$ is denoted by $G[V']$. 
A \emph{bipartite graph} whose partition has the parts $X$ and $Y$ is denoted by $G=(X,Y,E)$.
Note that if a bipartite graph $G=(X,Y,E)$ is complete, then $E=X\times Y$.
Finally, a \emph{Hamiltonian cycle} in $ G$ is a cycle that visits  every vertex exactly once.
A graph $G$ is a $(6,2)$ \emph{chordal graph} if  every cycle of $G$ of length at least six has at least two chords.


\begin{definition}
A graph $G=(V,E)$ is a \emph{permutation graph} if and only if there is an ordering $v_1,v_2, \dotsc, v_n$ of the vertices of $V$ such that there is permutation $\tau$ of the numbers from $1$ to $n$ with the property that for all integers $1\leq i<j\leq n$
$$(v_i,v_j)\in E \text{ if and only if }
\tau(i)>\tau(j).$$
\end{definition}

\begin{remark}\label{0409a}
Every induced subgraph of a permutation graph is a permutation graph. Also, a graph is a permutation graph if and only if every of its connected components is a permutation graph.  
\end{remark}

Finally, if $<$ is a (linear) ordering on a set $A$ and $X,Y\subseteq A$, then $X< Y$ means  that $x<y$ for every $x\in X$ and $y\in Y$.

\section{Basics on Parikh Word Representable Graphs}

\begin{definition}\cite{bera2016structural}\label{2409}
Suppose $\Sigma=\{a_1<a_2<\dotsb < a_s\}$ is an ordered alphabet and $w=w_1w_2\dotsm w_n$ is a nonempty word of length $n$ over $\Sigma$ with $w_i\in\Sigma$ for all $1\leq i\leq n$. 
Define the \emph{Parikh graph of $w$ over $\Sigma$}, denoted $\mathcal{G}(w)$, with set of vertices $\{1,2,3, \dotsc, n\}$ and for all $1\leq i<j \leq n$, the vertices $i$ and $j$ are adjacent if and only if for some $1\leq k\leq s-1$, we have $w_i=a_k$ and $w_j=a_{k+1}$. 
\end{definition}

In other words, to every occurrence of the  subword $a_ka_{k+1}$ in $w$, where $1\leq k \leq s-1$, there is an edge between the corresponding vertices in $\mathcal{G}(w)$.

\begin{remark}\label{0910a}
In \cite{teh2015coreb} $\core_{\Sigma}(w)$ is defined to be the unique subword of $w$ consisting of letters that contribute to the value of $\vert w\vert_{a_ka_{k+1}}$ for some
$1\leq k\leq s-1$. Hence, the definition of Parikh graph of $w$ resembles that of $\core_{\Sigma}(w)$. In fact, it is easy to see that 
$\mathcal{G}(\core_{\Sigma}(w))$ is isomorphic to the subgraph of $\mathcal{G}(w)$ induced by the set of non-isolated vertices. 
\end{remark}

According to the definition, over a given ordered alphabet, the Parikh graph of a word is unique. However, distinct words may have isomorphic Parikh graphs. For example, $\mathcal{G}(abb)$ and $\mathcal{G}(abc)$ are isomorphic over the ordered alphabet $\{a<b<c\}$.

\begin{example}
Over the ordered alphabet $\{a<b<c<d\}$, the Parikh graph of the word $bbccabdc$ is:
$$
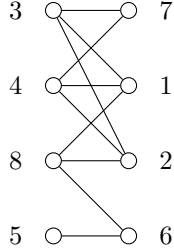
\begin{tikzpicture}
  \vertex (d) at (1,1){};
  \vertex (c1) at (0,1){};
  \vertex (c2) at (0,0){};
	\vertex (c3) at (0,-1){};
	\vertex (b1) at (1,0){};
	\vertex (b2) at (1,-1){};
	\vertex (b3) at (1,-2){};
	\vertex (a) at (0,-2){};
\node at (1.5,1) {\footnotesize 7};
\node at (-0.5,1) {\footnotesize 3};
\node at (-0.5,0) {\footnotesize 4};
\node at (-0.5,-1) {\footnotesize 8};
\node at (1.5,0) {\footnotesize 1};
\node at (1.5,-1) {\footnotesize 2};
\node at (1.5,-2) {\footnotesize 6};
\node at (-0.5,-2) {\footnotesize 5};
	\path
		(d) edge (c1)
		(d) edge (c2)
		(b1) edge (c1)
		(b1) edge (c2)
		(b1) edge (c3)
		(b2) edge (c1)
		(b2) edge (c2)
		(b2) edge (c3)
		(b3) edge (c3)
		(b3) edge (a)
	;
\end{tikzpicture}$$
\end{example}

\begin{remark}
Any Parikh graph of any word over any ordered alphabet is bipartite. This is because, by definition, if any two vertices of $\mathcal{G}(w)$
are adjacent, then
they correspond to letters $a_i$ in $w$ having subcripts with opposite parity. 
\end{remark}	

In \cite{bera2016structural} every Parikh graph of a binary word was indirectly shown to be a permutation graph by using the fact that a bipartite graph is a permutation graph if and only if its bipartite complement is a comparability graph. Here, we show this known result by directly producing the permutation.

\begin{proposition}\label{1307a}
	Suppose $\Sigma=\{a<b\}$ and
	$w=w_1w_2\dotsm w_n\in \Sigma^*$, where \linebreak $w_1,w_2, \dotsc, w_n\in \Sigma$. Let
	$\tau\colon \{1,2,\dotsc, n\}\rightarrow \{1,2,\dotsc, n\}$ be defined by
	$$\tau(x)=\begin{cases}
	i  &\text{ if }w_x=b \text{ and }   x= \pos_b(w,i) \\
	j+\vert w\vert_b  &\text{ if }w_x=a \text{ and } x=\pos_a(w,j). 
	\end{cases}
	$$
	Then for all integers $1\leq x<y\leq n$, the vertices $x$ and $y$ are adjacent in	
	$\mathcal{G}(w)$ if and only if $\tau(x)>\tau(y)$.
\end{proposition}

\begin{proof}
Suppose $ x$ and $y$ are integers such that $ 1 \leq x < y \leq n $. Suppose $ w_x = a $ and $ w_y = b $. By Definition~\ref{2409}, $ x $ and $ y $ are adjacent in $\mathcal{G}(w)$. Note that $ \tau (x)>\vert w \vert_b $ and $ \tau (y) \leq \vert w\vert_b$. Hence, $ \tau (x) > \tau (y)$.		
The case $ w_x = b $ and $ w_y = a $ is similar. Now, suppose $ w_x =  w_y = a $.
	By Definition~\ref{2409}, $ x $ and $ y $ are  not adjacent in $\mathcal{G}(w)$. 
	Suppose $x=\pos_a(w,j)$ and $y=\pos_a(w,j')$. Since $x<y$, it follows that $j<j'$.  
	Hence, $ \tau (x)= j + \vert w \vert_b <j'+\vert w\vert_b=\tau(y)$.		
The case $w_x=b$ and $w_y=b$ is similar.
\end{proof}

\begin{definition}
A bipartite graph $G$ is \emph{Parikh word representable} if and only if $G$ is isomorphic to $\mathcal{G}(w)$
for some word $w$ over some ordered alphabet $\Sigma$.
We say that $G$ is \emph{Parikh $n$-ary word representable} if and only if $\vert \Sigma\vert=n$. 
\end{definition}

In particular, $G$ is Parikh ternary word representable if and only if $G$ is isomorphic to $\mathcal{G}(w)$ for some word $w$ over the ordered alphabet $\{a<b<c\}$. Also, it is clear that if
$G$ is Parikh $n$-ary word representable, then 
$G$ is Parikh $m$-ary word representable for every $m\geq n$.

Up to isomorphism, the actual vertices of a Parikh graph is inessential. Hence, for convenience of this work, we propose the following variation and equivalent version of Parikh graphs.

\begin{definition}
Suppose $\Sigma=\{a_1<a_2<\dotsb < a_s\}$ is an ordered alphabet and $w=w_1w_2\dotsm w_n$ is a nonempty word of length $n$ over $\Sigma$ with $w_i\in\Sigma$ for all $1\leq i\leq n$. 
Define the \emph{Parikh graph of $w$ over $\Sigma$}, denoted $\mathcal{G}(w)= (V,E)$, with the set of vertices 
$$V=\{ \, (a_i, l) \mid 1\leq i\leq s, \vert w\vert_{a_i}\geq 1,
\text{ and }  1\leq l\leq \vert w\vert_{a_i}  \,\}$$
 and for all $(a_i,l), (a_{i'},l')\in V$, there is an (undirected) edge between them if and only if $$    \vert i-i'\vert=1 \text{ and } (i-i')(\pos_{a_i}(w,l)-\pos_{a_{i'}}(w,l'))>0.$$
\end{definition}

In other words, $(a_i,l)$ and $(a_{i+1},l')$ are adjacent if and only if $\pos_{a_i}(w,l)<\pos_{a_{i+1}}(w,l')$ and no other pair of vertices are adjacent. Note also that there is a one-to-one correspondence between the letters in $w$ and the vertices of $\mathcal{G}(w)$.

\begin{example}\label{2208a}
Over the ordered alphabet $\{a<b<c<d\}$, the Parikh graph of the word $bbccabdc$ is:
$$
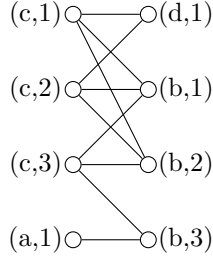
\begin{tikzpicture}
  \vertex (d) at (1,1){};
  \vertex (c1) at (0,1){};
  \vertex (c2) at (0,0){};
	\vertex (c3) at (0,-1){};
	\vertex (b1) at (1,0){};
	\vertex (b2) at (1,-1){};
	\vertex (b3) at (1,-2){};
	\vertex (a) at (0,-2){};
\node at (1.5,1) {\footnotesize (d,1)};
\node at (-0.5,1) {\footnotesize (c,1)};
\node at (-0.5,0) {\footnotesize (c,2)};
\node at (-0.5,-1) {\footnotesize (c,3)};
\node at (1.5,0) {\footnotesize (b,1)};
\node at (1.5,-1) {\footnotesize (b,2)};
\node at (1.5,-2) {\footnotesize (b,3)};
\node at (-0.5,-2) {\footnotesize (a,1)};
	\path
		(d) edge (c1)
		(d) edge (c2)
		(b1) edge (c1)
		(b1) edge (c2)
		(b1) edge (c3)
		(b2) edge (c1)
		(b2) edge (c2)
		(b2) edge (c3)
		(b3) edge (c3)
		(b3) edge (a)
	;
\end{tikzpicture}$$
\end{example}

In our version, the arrangement of the vertices seems to be more systematic. This is no coincidence and we will see in Lemma~\ref{0409b} that this ordering is essential to show that every Parikh graph is a bipartite permutation graph. Example~\ref{2208a} was chosen because it was given as an example of a bipartite permutation graph which is not Parikh binary word representable. 

\begin{remark}\cite[Lemma 4]{bera2016structural}\label{1807a} If $H$ is the subgraph of $\mathcal{G}(w)=(V,E)$ induced by a set of vertices $V'\subseteq V$, then $H$ is isomorphic to $\mathcal{G}(u)$ where $u$ is the subword of $w$ 
formed from the letters corresponding to the vertices in $V'$. In other words, every induced subgraph of a Parikh word representable graph is Parikh word representable.	
\end{remark}

Every connected component of a bipartite graph is obviously bipartite. Due to the following theorem, connectivity can be assumed in the study of Parikh word representability of bipartite graphs. 

\begin{theorem}\label{2906a}
	A bipartite graph is Parikh word representable if and only if every of its connected component is Parikh word representable.
\end{theorem}

\begin{proof}
The forward direction follows by Remark~\ref{1807a} because every connected component is an induced subgraph. 

Conversely, suppose $C_i$ for $1\leq i\leq l$ are the connected components of a bipartite graph $G$ and each $C_i$ is Parikh word representable, say $C_i$ is isomorphic to $\mathcal{G}(v_i)$ for some word $v_i$ over some ordered aphabet $\Pi_i$. Without loss of generality, we may assume that 
$\Pi_i\cap \Pi_j=\emptyset$ whenever $1\leq i<j\leq l$.
Let $w=v_1v_2\dotsm v_l$
and let $\Sigma$ denote the ordered alphabet with $\bigcup_{i=1}^l \Pi_i$ as underlying alphabet 
such that its ordering is defined as follows:
for every $x,y\in \bigcup_{i=1}^l \Pi_i$,
$x<y$ if and only if either
\begin{enumerate}
\item $x\in \Pi_j$ and $y\in \Pi_{j'}$ for some $1\leq j'<j\leq l$; or
\item $x,y\in \Pi_j$ for some $1\leq j\leq l$ and 
$x<y$ in the ordering of $\Pi_j$.
\end{enumerate}
Then it can be verified that $G$ is isomorphic to $\mathcal{G}(w)$.
\end{proof}

\begin{remark}\label{2110a}
If a connected bipartite graph $G$ 
is Parikh word representable, then it is Parikh representable by a word $w$ (meaning $G$ is isomorphic to $\mathcal{G}(w)$) over some ordered alphabet $\Sigma$ such that $\supp(w)=\Sigma$. 
\end{remark}

We end this section by presenting a nice connection between Parikh graphs and partitions. A word $w$ over an alphabet $\Sigma$ is said to be \emph{slender} if and only if $\vert w\vert_a=1$ for all $a\in \Sigma$. Note that the Parikh graph of any slender word is simply a disjoint union of paths and isolated vertices.

From now on, let $\Sigma_s$ denote the fixed ordered alphabet $\{a_1<a_2<\dotsb<a_s\}$ for every integer $s\geq 2$.

\begin{theorem}
For every integer $s\geq 2$, the number of distinct bipartite graphs (up to isomorphism) Parikh word representable by slender words over $\Sigma_s$ is the number of distinct partitions of $s$.	
\end{theorem}

\begin{proof}
Suppose $s\geq 2$ and $w$ is a slender word over $\Sigma_s$.
Each isolated vertex in $\mathcal{G}(w)$ can be viewed as a degenerate path. Then $\mathcal{G}(w)$
is simply a disjoint union of paths and can be associated to a partition of $s$ whose respective summand corresponds to the number of vertices in the respective path.

Note that the Parikh graphs of any two distinct slender words over $\Sigma_s$ are isomorphic if and only if they are  associated to the same partition of $s$.
Conversely, any partition of $s$ can be associated to $\mathcal{G}(w)$ for some slender word $w$ over $\Sigma_s$. Therefore, the number of distinct bipartite graphs Parikh word representable by slender words over $ \Sigma_s $ is the number of partition of $ s $.
\end{proof}

\begin{example}
Up to isomorphism,
$\mathcal{G}(abcd)$, $\mathcal{G}(bcda)$,
$\mathcal{G}(cdab)$, $\mathcal{G}(cdba)$, and\linebreak $\mathcal{G}(dcba)$  are all the distinct bipartite graphs Parikh word representable by slender words over $\{a<b<c<d\}$, corresponding respectively to the partitions $4=3+1=2+2=2+1+1=1+1+1+1$ of the integer four.	
\end{example}

\section{As Bipartite Permutation graphs}

In \cite{bera2016structural} it was already shown that every graph Parikh representable by a \emph{binary} word is a bipartite permutation graph. Using the following characterization of bipartite permutation graphs, we will first show that this in fact is true for every Parikh word representable graph. The remainder of this section after that is devoted to prove the converse. Together, we obtain our main result that says that the two classes of bipartite graphs coincide.

\begin{definition}
	Suppose $G=(X,Y,E)$ is a bipartite graph. A \emph{strong ordering} on the vertices of $G$ is an ordered pair $(<_X, <_Y)$ where $<_X$ is an
	ordering on $X$ and 
	$<_Y$ is an ordering on $Y$	
	such that for all $(x,y), (x',y')\in E$, if $x<_X x'$ and $y'<_Y y$, then $(x,y')$ and $(x',y)$ are both in $E$.	
\end{definition}

\begin{theorem}\cite{spinrad1987bipartite}\label{1207a}
	Suppose $G=(X,Y,E)$ is a bipartite graph. Then $G$ is a permutation graph if and only if there exists a strong ordering on the vertices of $G$.	
\end{theorem}

\begin{lemma}\label{0409b}
Suppose $s\geq 2$ and $w\in \Sigma_s^*$ with $\supp(w)=\Sigma_s$. Let $\mathcal{G}(w)=(X,Y,E)$ where $X$ (respectively $Y$) consists of vertices of the form $(a_i,l)$ where $i$ is odd (respectively even). Let $<_X$ denote the following ordering on $X$:	
\begin{multline*}
(a_{2\lceil s/2\rceil-1}, 1),(a_{2\lceil s/2\rceil-1}, 2), \dotsc, (a_{2\lceil s/2\rceil-1}, \vert w\vert_{ a_{2\lceil s/2\rceil-1}   }),\dotsc,\\
(a_3, 1), (a_3,2)\dotsc, (a_3, \vert w\vert_{a_3}),(a_1, 1),(a_1,2), \dotsc, (a_1, \vert w\vert_{a_1});
\end{multline*}
and let $<_Y$ denote the following ordering on $Y$:	
	\begin{multline*}
(a_{2\lfloor s/2\rfloor}, 1),(a_{2\lfloor s/2\rfloor}, 2), \dotsc, (a_{2\lfloor s/2\rfloor}, \vert w\vert_{ a_{2\lfloor s/2\rfloor}   }),\dotsc,\\
(a_4, 1), (a_4,2)\dotsc, (a_4, \vert w\vert_{a_4}),(a_2, 1),(a_2,2), \dotsc, (a_2, \vert w\vert_{a_2}),
\end{multline*}
where $\lceil \phantom{2}\rceil$ and $\lfloor \phantom{2} \rfloor$ are the ceiling and floor functions respectively.
Then $(<_X,<_Y)$ is a strong ordering on the vertices of $\mathcal{G}(w)$.	
\end{lemma}

\begin{proof}
	Suppose $(x,y), (x',y')\in E$ and $x<_X x'$ and $y'<_Y y$. We need to show that $(x,y')$ and $(x',y)$ are both in $E$.
	
	\setcounter{case}{0}
	
	\begin{case} $x=(a_i,l)$ and $x'=(a_i,l')$ for some (odd) $i$ and $l<l'$.
	
	The following subcases exhaust all possibilities because $(x,y), (x',y')\in E$.	
		\begin{subcase} $y=(a_{i+1},m)$ and $y'=(a_{i+1},m')$ for some $m>m'$.
			
		Since $(x',y')\in E$, by  definition, $\pos_{a_{i+1}}(w, m')>\pos_{a_i}(w,l')$. Hence,
			$$\pos_{a_{i}}(w, l)<\pos_{a_i}(w,l')<\pos_{a_{i+1}}(w, m')<\pos_{a_{i+1}}(w,m).$$
			It follows that $(x,y')\in E$ and $(x',y)\in E$ by the definition of Parikh graph.
		\end{subcase}
		
		\begin{subcase}$y=(a_{i-1},m)$ and $y'=(a_{i-1},m')$ for some $m>m'$.
			
			This is similar to Case~1.1.
		\end{subcase}
		
		\begin{subcase} $y=(a_{i-1},m)$ and $y'=(a_{i+1},m')$ for some $m$ and $m'$.
			
	Since $(x,y),(x',y')\in E$, by the definition of Parikh graph, $\pos_{a_{i+1}}(w, m')>\pos_{a_i}(w,l')$
			and $\pos_{a_{i-1}}(w, m)<\pos_{a_i}(w,l)$. Hence,
			$$\pos_{a_{i-1}}(w, m)<\pos_{a_i}(w,l)<\pos_{a_i}(w, l')<\pos_{a_{i+1}}(w,m').$$
			It follows that $(x,y')\in E$ and $(x',y)\in E$.
		\end{subcase}
	\end{case}
	
	\begin{case} $x=(a_i,l)$ and $x'=(a_{i'},l')$ for some $l$, $l'$, and (odd) $i>i'$.
		
		Since $(x,y)\in E$ and $(x',y')\in E$, it follows that
		$i-i'$ cannot be more than two for otherwise $y<_Y y'$, a contradiction. Hence, in fact, $x'=(a_{i-2},l')$,
		$y=(a_{i-1},m)$ and $y'=(a_{i-1},m')$ for some $m>m'$.
	Similarly, $\pos_{a_{i-1}}(w, m')>\pos_{a_{i-2}}(w,l')$
		and $\pos_{a_{i-1}}(w, m)<\pos_{a_{i}}(w,l)$. Hence,
		$$\pos_{a_{i-2}}(w, l')<\pos_{a_{i-1}}(w,m')<\pos_{a_{i-1}}(w, m)<\pos_{a_{i}}(w,l).$$
		It follows that $(x,y')\in E$ and $(x',y)\in E$.\qedhere
	\end{case}
\end{proof}

\begin{theorem}\label{1207b}
	Every Parikh word representable graph is a bipartite permutation graph.	
\end{theorem}

\begin{proof}
	Suppose $G$ is Parikh word representable. By Remarks~\ref{0409a} and \ref{1807a}, we may assume $G$ is connected.  Without loss of generality and by Remark~\ref{2110a}, we may assume that $G=\mathcal{G}(w)$ for some word $w\in \Sigma_s$ with $\supp(w)=\Sigma_s$.
By Lemma~\ref{0409b}, there exists a strong ordering on the vertices of $G$. Therefore, $G$ is a bipartite permutation graph by Theorem~\ref{1207a}. 
\end{proof}

We need a technical lemma before our main theorem.

\begin{lemma}\label{2508a}
	Suppose $G=(X,Y,E)$ is a connected bipartite permutation graph. Let $(<_X, <_Y)$ be a strong ordering on the vertices of $G$. Then 
	for some positive integer $n$,
	there exist a sequence $X_1, X_2, \dotsc, X_n$ of nonempty intervals of $X$ (with respect to $<_X$) with $\bigcup_{i=1}^n X_i=X$
	 and a sequence $Y_1, Y_2, \dotsc, Y_n$ of nonempty intervals of $Y$ (with respect to $<_Y$) with $\bigcup_{i=1}^n Y_i=Y$
	such that 
	\begin{enumerate}
		\item $\bigcup_{i=1}^p X_i$ is an end segment of $X$ and $X_p$ is an is an initial segment of 
	$\bigcup_{i=1}^p X_i$ for each $1\leq p\leq n$;
	\item $\bigcup_{i=1}^p Y_i$ is an end segment of $Y$ and $Y_p$ is an is an initial segment of 
		$\bigcup_{i=1}^p Y_i$ for each $1\leq p\leq n$;	
	\item $X_p \nsubseteq X_{p+1}$ and
$Y_p \nsubseteq Y_{p+1}$ for each $1\leq p\leq  n-1$;	
	\item at least one of 
	$X_{p+1}\backslash X_p$
	or $Y_{p+1}\backslash Y_p$ is nonempty 
for each $1\leq p\leq n-1$;	
\item $E(G_{p})=E(G_{p-1})\cup  (X_p\times Y_p)$ for each $1\leq p\leq n$,
\end{enumerate}
where $G_p= G[\bigcup_{i=1}^p X_i \cup \bigcup_{i=1}^p Y_i ]$ for each $1\leq p \leq n$ and $E(G_0)=\emptyset$.
\end{lemma}

\begin{proof}
We will recursively construct the required two finite sequences that in fact satisfy the following two additional properties:
\begin{itemize}
\item[(6)] $X_p$ is the set of  vertices in $X$ adjacent to the last vertex of $Y_p$ but preceding or equal to the last vertex of $X_p$ for each $1\leq p\leq n$;
\item[(7)] $Y_p$ is the set of vertices in $Y$ adjacent to the last vertex of $X_p$ but preceding or equal to the last vertex of $Y_p$ for each $1\leq p\leq n$.
\end{itemize}	

For the base step, let	$X_1$ be the set of vertices adjacent to the last vertex $y_L$ of $Y$
	 and let
	$Y_1$ be the set of vertices adjacent to the last vertex $x_L$ of $X$.
	It is clear that 
	$X_1$ and $Y_1$ are nonempty because $G$ is connected.
	
	Assume $X_1$ is not an end segment of $X$. Then there exists $x'\in X_1$ 
	and $x\in X\backslash X_1$  
	such that $x'<_X x$. By definition,
	$(x',y_L)\in E$. Since $G$ is connected and $(x, y_L)\notin E$, it follows that $(x, y)\in E$ for some $y<_Y y_L$. The definition of strong ordering implies tha $(x, y_L)\in E$, a contradiction.
	Hence, $X_1$ is an end segment of $X$. Similarly, it can be shown that $Y_1$ is an end segment of $Y$.
	
To complete the base step, it remains to show that 
	$(x,y)\in E$ whenever $x\in X_1$ and $y\in Y_1$ for then $E(G_1)= E(G[X_1\cup Y_1])= X_1\times Y_1$. Suppose $x\in X_1$ and $y\in Y_1$.
	If $x=x_L$ or $y=y_L$, we are done. Assume $x\neq x_L$ and $y\neq y_L$.
	By definition, 
	$(x, y_L), (x_L, y)\in E$.
	Since $x<_X x_L$ and $y <_Y y_L$, similarly,  $(x, y)\in E$ due to the strong ordering.
	
	For the recursive step, suppose a sequence $X_1, X_2, \dotsc, X_l$ 
	of nonempty intervals of $X$
	and a sequence $Y_1, Y_2, \dotsc, Y_l$ 
	of nonempty intervals of $Y$ 
	have been constructed such that properties (1)--(7) hold with $n$ replaced by $l$. 
 It suffices to show that our recursive construction can be continued as long as 
	$\bigcup_{i=1}^l X_i\neq X$ or	$\bigcup_{i=1}^l Y_i\neq Y$.
		
	Assume at least one of $X\backslash \bigcup_{i=1}^l X_i$ or $Y\backslash \bigcup_{i=1}^l Y_i$ is nonempty.
	Let $x^\ast$ be the last vertex in $X$
	incident on an edge not in $E(G_l)$ and let
	$y^\ast$ be the last vertex in $Y$
		incident on an edge not in $E(G_l)$.
		Since $G_l\neq G$ and  $G$ is connected, $x^\ast$ and $y^\ast$ are well-defined.
	Let $X_{l+1}$ be the set of  vertices in $X$ adjacent to $y^\ast$ but preceding or equal to $x^\ast$
			 and vice versa let
			$Y_{l+1}$ be the set of vertices in $Y$ adjacent to $x^\ast$ but preceding or equal to $y^\ast$. By the definitions, it can be verified that both $X_{l+1}$ and $Y_{l+1}$ are nonempty. Furthermore, it can be argued as in the base step that $X_{l+1}$ is an interval of $X$ with last vertex $x^\ast$, $Y_{l+1}$ is an interval of $Y$ with last vertex $y^\ast$, and $X_{l+1}\times Y_{l+1}\subseteq E$. Hence, properties (6) and (7)
			are true for the instance $p=l+1$. 

Note that if $x^\ast \notin \bigcup_{i=1}^l X_i$, then it is the last vertex 
in $X\backslash \bigcup_{i=1}^l X_i$.
Let the last vertex of $X_l$ (respectively $Y_l$) be denoted by $x_{L,l}$ (respectively $y_{L,l}$).
Suppose $x^\ast \in \bigcup_{i=1}^l X_i$.
We argue that $x^\ast$ precedes $x_{L,l}$ by contradiction. Assume $x_{L,l}\leq_X x^\ast$. Since $x^\ast$ is incident on an edge not in $E(G_l)$, 
it follows that $x^\ast$ is adjacent to some vertex
$y\in Y\backslash \bigcup_{i=1}^l Y_i$.
However, $y<_Y y_{L,1}$ and $(x_{L,1}, y_{L,1})\in E$. Hence, due to the strong ordering or simply because $x^\ast=x_{L,l}$, it follows that $x_{L,1}$ is adjacent to $y$. However, this would contradict the fact that $Y_l$ is the set of vertices in $Y$ adjacent to $x_{L,l}$ but preceding or equal to $y_{L,l}$. Therefore, $x^\ast$ either belongs to $X_l\backslash \{x_{L,l}\}$ or is the last vertex 
in $X\backslash \bigcup_{i=1}^l X_i$. Similarly,  
$y^\ast$ either belongs to $Y_l\backslash \{y_{L,l}\}$ or is the last vertex 
in $Y\backslash \bigcup_{i=1}^l Y_i$.

Now, using the strong ordering and the fact that $X_{l}\times Y_{l}\subseteq E$, it can be shown that if $x^\ast\in X_l\backslash \{x_{L,l}\}$, then
$\{\,x\in X_L\mid x\leq_X x^\ast \,\} \subseteq X_{l+1}$. Hence, since $X_{l+1}$ is an interval of $X$ with last vertex $x^\ast$, regardless of whether 
$x^{\ast}$ belongs to $X_l\backslash \{x_{L,l}\}$ or is the last vertex 
in $X\backslash \bigcup_{i=1}^l X_i$,
it follows that $X_l\nsubseteq X_{l+1}$, $\bigcup_{i=1}^{l+1} X_i$ is an end segment of $X$ and $X_{l+1}$ is an initial segment of $\bigcup_{i=1}^{l+1} X_i$. Similarly, it can be shown that $Y_l\nsubseteq Y_{l+1}$   $\bigcup_{i=1}^{l+1} Y_i$ is an end segment of $Y$ and $Y_{l+1}$ is an initial segment of $\bigcup_{i=1}^{l+1} Y_i$.
Furthermore, if $x^\ast\in X_l \backslash \{x_{L,l}\}$, then it is adjacent to some $y\in Y_{l+1}\backslash \bigcup_{i=1}^l Y_i
=Y_{l+1}\backslash Y_l$. Otherwise, $x^\ast$ is the last vertex in $X_{l+1}\backslash \bigcup_{i=1}^l X_i= X_{l+1}\backslash X_l$. Therefore, at least one of  $X_{l+1}\backslash X_l   $ or $Y_{l+1}\backslash Y_l$ is nonempty.

Finally, to complete the recursion step, it remains to show that	
$E(G_{l+1})=E(G_{l})\cup (X_{l+1}\times Y_{l+1})$. Since $X_{l+1}\times Y_{l+1}\subseteq E$,
it suffices to show that $x\in X_{l+1}$ and $y\in Y_{l+1}$ whenever $(x,y)\in E(G_{l+1})\backslash E(G_l)$. Suppose not. Then we may assume that
$(x,y)\in E(G_{l+1})\backslash E(G_l)$ for some
$x\in X_{l+1}\backslash \bigcup_{i=1}^l X_i$ and $y\in (\bigcup_{i=1}^l Y_i) \backslash Y_{l+1}$ as the other case is  similar. Note that $y^\ast<_Y y$ because $Y_{l+1}$ is an initial segment of $\bigcup_{i=1}^{l+1} Y_i$ and $y^\ast \in Y_{l+1}$. However, since $(x,y)\notin E(G_l)$, this contradicts the definition of $y^\ast$. 
\end{proof}

\begin{remark}
The sequence constructed in the proof of Lemma~\ref{2508a} in fact satisfies and additional property that at least one of 
	$X_{p} \cap X_{p+1}$
	or $Y_{p}\cap Y_{p+1}$ is nonempty 
for each $1\leq p\leq n-1$. However, this fact is not needed in the proof of the next main theorem.
\end{remark}

\begin{theorem}\label{0309c}
	Every bipartite permutation graph is Parikh word representable.
\end{theorem}

\begin{proof}
	By Theorem~\ref{2906a} and Remark~\ref{0409a}, it suffices to consider only connected bipartite permutation graphs. Suppose $G=(X,Y,E)$ is a connected bipartite permutation graph. Let $(<_X, <_Y)$ be a strong ordering on the vertices of $G$. Let $X_1, X_2, \dotsc, X_n$ and $Y_1, Y_2, \dotsc, Y_n$ for some $n$ be sequences given by Lemma~\ref{2508a}. It can be deduced using properties  $(1)$ and $(3)$ 
	in Lemma~\ref{2508a}
	that $X_l\cap X_{l+1}$  is an initial segment of $X_{l}$ and $X_{l+1}\backslash X_l <_X
	\bigcup_{i=1}^l X_i$	
	for each $1\leq l \leq n-1$. 
	(It is possible for $X_l\cap X_{l+1}$
	or $X_{l+1}\backslash X_l$ to be empty but
	$X_{l+1}$ is nonempty.)	
	Similarly $Y_l\cap Y_{l+1}$ is an initial segment of $Y_l$  
	and $Y_{l+1}\backslash Y_l <_Y
		\bigcup_{i=1}^l Y_i$ for each $1\leq l \leq n-1$.	
	 
	 For each $1\leq l \leq n$, let $G_l= G[\bigcup_{i=1}^l X_i \cup \bigcup_{i=1}^l Y_i ]$.
	Recursively, we will construct a sequence $w_1, w_2, \dotsc, w_n$ of words such that
	\begin{enumerate}
		\item $w_l$ is a proper subword of $w_{l+1}$ for each $1\leq l\leq n-1$;
		\item $\supp(w_l)=\Sigma_s$ for some $s$ for each $1\leq l\leq n$  ;
		\item For each $1\leq l\leq n$, there is an isomorphism $\varphi_l \colon V(G_l) \rightarrow V(\mathcal{G}(w_l) )$ between $G_l$ and $\mathcal{G}(w_l)$ that preserves the strong ordering
		on the vertices of $G_l$ inherited from $(<_X, <_Y)$ and the  strong ordering on the vertices of $\mathcal{G}(w_l)$ as given in Lemma~\ref{0409b}.
		Furthermore, under this isomorphism, the subword of $w_l$ that corresponds to $X_l\cup Y_l$ is either binary or ternary.
	\end{enumerate}
Since $G_n=G$, if our construction can be carried out, then $\varphi_n$ will be an isomorphism between $G$ and $\mathcal{G}(w_n)$ and thus $G$ is Parikh word representable.

	For the basis step, since $E(G_1)= X_1\times Y_1$, we can take $w_1=a_1^{\vert X_1\vert}a_2^{\vert Y_1\vert}$. There is a canonical isomorphism $\varphi_1$ between $G_1$ and $\mathcal{G}(w_1)$ with the required properties.
	
	For the recursion step, suppose $w_l$ has been constructed with $\supp(w_l)=\Sigma_s$ for some $s$ and let
	$\varphi_l \colon V(G_l) \rightarrow V(\mathcal{G}(w_l) )$ be an isomorphism 
	between $G_l$ and $\mathcal{G}(w_l) $
	with the stated property. 	
	
	\setcounter{case}{0}

	\begin{case}\label{2708a} The subword of $w_l$ that corresponds to $X_l\cup Y_l$ is binary. 
		
Since $X_l$ is an initial segment of $\bigcup_{i=1}^l X_i$ and $Y_l$ is an initial segment of $\bigcup_{i=1}^l Y_i$, by the strong ordering preserving property of the isomorphism $\varphi_l$, we may assume that
		$\varphi_l[X_l]= \{(a_{s-1},1), (a_{s-1},2), \dotsc, (a_{s-1}, \vert X_l\vert)\}$ and $\varphi_l[Y_l]=\{(a_{s},1),  \dotsc, (a_{s}, \vert Y_l\vert)\}$ as the other case is similar. 
		Since $X_l\times Y_l\subseteq E(G_l)$, due to the isomorphism, it follows that
		every letter in $w_l$ that is $a_s$  comes after 
		the $\vert X_l\vert$-th $a_{s-1}$  in $w_l$.

		Let $w_{l+1}$ be obtained from $w_l$ by inserting $a_s^{\vert Y_{l+1}\backslash Y_l\vert}$ immediately after the $\vert X_l\cap X_{l+1}\vert$-th $a_{s-1}$ in $w_l$
		(in case $X_l\cap X_{l+1}$ is empty,  immediately before the first $a_{s-1}$ in $w_l$)
		 and inserting $a_{s+1}^{ \vert X_{l+1}\backslash X_l\vert}$ immediately after the $\vert Y_l\cap Y_{l+1}\vert$-th $a_{s}$ in $w_l$ (in case $Y_l\cap Y_{l+1}$ is empty, immediately before the first $a_{s}$ in $w_l$). 
		 From the previous observation, note that the $a_s^{\vert Y_{l+1}\backslash Y_l\vert}$ newly  introduced in $w_{l+1}$ comes before every $a_s$ originally in $w_l$.

		 Since at least one of  $X_{l+1}\backslash X_l$ or $Y_{l+1}\backslash Y_l$ is nonempty, $w_l$ is a proper subword of $w_{l+1}$. Furthermore, $\supp(w_{l+1})$ is equal to $\Sigma_s$ or $\Sigma_{s+1}$.

Let $x_1, x_2, \dotsc, x_{\vert X_{l+1}\backslash X_l\vert}$ be the vertices in  $X_{l+1}\backslash X_l$ (provided it is nonempty) ordered according to $<_X$. Similarly, let $y_1, y_2, \dotsc, y_{\vert Y_{l+1}\backslash Y_l\vert}$ be the vertices in  $Y_{l+1}\backslash Y_l$, ordered according to $<_Y$.
Consider the mapping $\varphi_{l+1}\colon V(G_{l+1})     )\rightarrow V( \mathcal{G}(w_{l+1})  )$ defined by
	$$\varphi_{l+1}(z)=\begin{cases}
	\varphi_l(z)  &\text{if } z\in V(G_l) \text{ and } \varphi_l(z)= (a_i,j)\\
&\phantom{\text{if}} \text{ for some } 1\leq i \leq s-1 \text{ and } 1\leq j \leq \vert w_l\vert_{a_i};\\
(a_s, j'+\vert Y_{l+1}\backslash Y_l\vert  )&\text{if } z\in V(G_l) \text{ and }   \varphi_l(z)= (a_s,j')\\
&\phantom{\text{if}}\text{ for some }  1\leq j' \leq \vert w_l\vert_{a_s};\\
(a_{s+1},t) &\text{if } z=x_t \text{ for some } 1\leq t\leq \vert X_{l+1}\backslash X_l\vert;\\
(a_{s},t') &\text{if } z=y_{t'} \text{ for some } 1\leq t' \leq \vert Y_{l+1}\backslash Y_l\vert.
	\end{cases} $$

Due to the property of $\varphi_l$ and by the definition of $\varphi_{l+1}$,  
 the restriction of $\varphi_{l+1}$ to $V(G_l)$ is an isomorphism between the subgraph $G_l$ of $G_{l+1}$ and the subgraph of $\mathcal{G}(w_{l+1})$ induced by the subword $w_l$ of $w_{l+1}$ that preserves the corresponding strong orderings. Since
 $X_{l+1}\backslash X_l<_X \bigcup_{i=1}^{l}X_i$ and $Y_{l+1}\backslash Y_l <_Y \bigcup_{i=1}^{l}Y_i $, it follows that $\varphi_{l+1}$
 preserves the strong ordering
 		on the vertices of $G_{l+1}$ inherited from $(<_X, <_Y)$ and the  strong ordering on the vertices of $\mathcal{G}(w_{l+1})$ as given in Lemma~\ref{0409b}.

 Since $X_l\cap X_{l+1}$ is an initial segment of $X_{l}$, note that 
 	\begin{align*}
  	\varphi_{l+1}[X_{l+1}]={}&
  	\varphi_{l+1}[X_{l+1}\backslash X_l]\cup
  	  \varphi_{l+1}[X_l \cap X_{l+1}   ]\\	
  ={}&	\{(a_{s+1}, 1),  \dotsc, (a_{s+1}, \vert X_{l+1}\backslash X_l\vert)\} \cup \{(a_{s-1},1), 
 	\dotsc, (a_{s-1}, \vert X_l\cap X_{l+1}\vert)\}.
 	\end{align*}
Also, $\varphi_{l}[Y_l\cap Y_{l+1}] =\{(a_{s},1), (a_{s},2), \dotsc, (a_{s}, \vert Y_l\cap Y_{l+1}\vert)\}$ because $Y_l\cap Y_{l+1}$ is an initial segment of $Y_{l}$
and thus  
$$\varphi_{l+1}[Y_l\cap Y_{l+1}] =
\{(a_{s}, 1+ \vert Y_{l+1}\backslash Y_l\vert ), 
(a_{s}, 2+ \vert Y_{l+1}\backslash Y_l\vert ),
\dotsc, (a_{s}, \vert Y_{l+1}\vert  )\}.$$ Hence,  note that 
$$
\varphi_{l+1}[Y_{l+1}]={}
  	\varphi_{l+1}[Y_{l+1}\backslash Y_l]\cup
  	  \varphi_{l+1}[Y_l \cap Y_{l+1}   ]\\	
  ={} \{(a_{s},1), (a_s,2)  \dotsc, (a_{s}, \vert Y_{l+1}\vert)\}.$$	
 Furthermore, since $X_l\times Y_l\subseteq E(G_l)$, note that  
  $$E(G_{l+1})= 
  E(G_l)\cup (X_{l+1}\times Y_{l+1})=
   E(G_l)\cup \big[ (X_{l+1}\backslash X_l)\times Y_{l+1}\big] \cup 
\big[(X_{l+1}\times (Y_{l+1}\backslash Y_l)\big].
  $$

Therefore, to see that $\varphi_{l+1}$ is an isomorphism between $G_{l+1}$ and $\mathcal{G}(w_{l+1})$, it suffices to show that 
 for the graph $\mathcal{G}(w_{l+1})$ 
 \begin{enumerate}
 \item $N\big( (a_{s+1}, t) \big)= \varphi_{l+1}[Y_{l+1}]$  for each $1\leq t\leq \vert X_{l+1}\backslash X_l\vert$; and
 \item $N\big( (a_{s}, t') \big)= \varphi_{l+1}[X_{l+1}]$  for each $1\leq t'\leq \vert Y_{l+1}\backslash Y_l\vert$.
  \end{enumerate}

Since
the $a_s^{\vert Y_{l+1}\backslash Y_l\vert}$ newly introduced in $w_{l+1}$ comes before every $a_s$ originally in $w_l$, the $a_{s+1}^{\vert X_{l+1}\backslash X_l\vert}$ newly introduced in $w_{l+1}$ comes immediately after the \mbox{$\vert Y_{l+1}\vert$-th} $a_s$ in $w_{l+1}$ because $\vert Y_{l+1}\backslash Y_l\vert+ \vert Y_l \cap Y_{l+1}\vert =\vert Y_{l+1}\vert$.
Hence, by the definition of Parikh graph,
$N\big( (a_{s+1}, t) \big)= \varphi_{l+1}[Y_{l+1}]$  for each $1\leq t\leq \vert X_{l+1}\backslash X_l\vert$.
 Meanwhile, since $a_s^{\vert Y_{l+1}\backslash Y_l\vert}$ is introduced in $w_{l+1}$ immediately after the $\vert X_{l}\cap X_{l+1}\vert$-th $a_{s-1}$ in $w_l$ and 
 they come
 before
 $a_{s+1}^{\vert X_{l+1}\backslash X_l\vert}$ and the rest of the $a_s$'s in $w_{l+1}$, it follows that
$N\big( (a_{s}, t') \big)= \varphi_{l+1}[X_{l+1}]$  for each $1\leq t'\leq \vert Y_{l+1}\backslash Y_l\vert$.

Finally, the first component of any vertex in $\varphi_{l+1}[X_{l+1}\cup Y_{l+1}]=
\varphi_{l+1}[X_{l+1}]\cup \varphi_{l+1}[Y_{l+1}] $
is either $a_{s-1}$, $a_s$, or $a_{s+1}$. 
Therefore, under the isomorphism $\varphi_{l+1}$, the subword of $w_{l+1}$ that corresponds to $X_{l+1}\cup Y_{l+1}$ is either binary or ternary, where being binary is possible because one of $ X_{l+1}\backslash X_l$ or $X_l \cap X_{l+1}$ can be empty but not both.
	\end{case}

	\begin{case} The subword of $w_l$ that corresponds to $X_l\cup Y_l$ is ternary. 
	
The proof of this case resembles that of Case 1 and thus some argument will be handwaived due to similarity.		
Since $X_l$ is an initial segment of $\bigcup_{i=1}^l X_i$ and $Y_l$ is an initial segment of $\bigcup_{i=1}^l Y_i$, by the strong ordering preserving property of the isomorphim $\varphi_l$, we may assume that $\varphi_l[X_l]=\{(a_{s-1},1), (a_{s-1},2), \dotsc, (a_{s-1}, \vert X_l\vert)\}$ and $$\varphi_l[Y_l]=\{(a_{s},1), (a_{s},2), \dotsc, (a_{s}, \vert w_l\vert_{a_s}),
		(a_{s-2},1), (a_{s-2},2), \dotsc, (a_{s-2}, \vert Y_l\vert- \vert w_l\vert_{a_s})\}.$$ 

\begin{subcase}
$\vert Y_l\cap Y_{l+1}\vert \leq \vert w_l\vert_{a_s}$.

This implies that $\varphi_l[Y_l\cap Y_{l+1}]= \{(a_{s},1), (a_{s},2), \dotsc, (a_{s}, \vert Y_l\cap Y_{l+1}\vert)\}$. Hence, we are essentially back to Case~1. We claim without going into details that the construction of $w_{l+1}$ and $\varphi_{l+1}$ as in Case~1 works here.
\end{subcase}

\begin{subcase}
$\vert Y_l\cap Y_{l+1}\vert > \vert w_l\vert_{a_s}$.

This implies that
\begin{multline*}
\varphi_l[Y_l\cap Y_{l+1}]= \{(a_{s},1), (a_{s},2), \dotsc, (a_{s}, \vert  w_l\vert_{a_s}),\\ (a_{s-2},  1), (a_{s-2},2), \dotsc,  (a_{s-2}, \vert Y_{l+1}\cap Y_l\vert -\vert w_l\vert_{a_s} )   \}.
\end{multline*}		

Let $w_{l+1}$ be the word obtained from $w_l$ by inserting $a_s^{ \vert Y_{l+1}\backslash Y_l\vert}$ immediately after the \linebreak $\vert X_l\cap X_{l+1}\vert$-th $a_{s-1}$ in $w_l$ and inserting also $a_{s-1}^{ \vert X_{l+1}\backslash X_l\vert}$ immediately after the \linebreak
\mbox{$(\vert Y_l\cap Y_{l+1}\vert -\vert w_l\vert_{a_s}) $-th} $a_{s-2}$ in $w_l$.

Let $x_1, x_2, \dotsc, x_{\vert X_{l+1}\backslash X_l\vert}$ and  $y_1, y_2, \dotsc, y_{\vert Y_{l+1}\backslash Y_l\vert}$ be defined as in Case~1.
Consider the mapping $\varphi_{l+1}\colon V(G_{l+1})     )\rightarrow V( \mathcal{G}(w_{l+1})  )$  defined by
	$$\varphi_{l+1}(z)=\begin{cases}
	\varphi_l(z)  &\text{if } z\in V(G_l) \text{ and } \varphi_l(z)= (a_i,j)\\
	&\phantom{\text{if}}  \text{ for some }
1\leq i \leq s-2 \text{ and } 1\leq j \leq \vert w_l\vert_{a_i};\\
(a_{s-1}, j'+\vert X_{l+1}\backslash X_l\vert  )&\text{if } z\in V(G_l)\text{ and } \varphi_l(z)= (a_{s-1},j')\\
&\phantom{\text{if}}\text{ for some }
1\leq j' \leq \vert w_l\vert_{a_{s-1}};\\
(a_s, j''+\vert Y_{l+1}\backslash Y_l\vert  )&\text{if } z\in V(G_l) \text{ and }\varphi_l(z)= (a_s,j'')\\
&\phantom{\text{if}}\text{ for some }
1\leq j'' \leq \vert w_l\vert_{a_s};\\
(a_{s-1},t) &\text{if } z=x_{t} \text{ for some } 1\leq t \leq \vert X_{l+1}\backslash X_l\vert;\\
(a_{s},t') &\text{if } z=y_{t'} \text{ for some } 1\leq t' \leq \vert Y_{l+1}\backslash Y_l\vert.
\end{cases} $$
Similarly, $\varphi_{l+1}$
 preserves the strong ordering
 		on the vertices of $G_{l+1}$ inherited from $(<_X, <_Y)$ and the  strong ordering on the vertices of $\mathcal{G}(w_{l+1})$ as given in Lemma~\ref{0409b}.
 
To see that $\varphi_{l+1}$ is an isomorphism between $G_{l+1}$ and $\mathcal{G}(w_{l+1})$, it suffices to show that 
 for the graph $\mathcal{G}(w_{l+1})$ 
 \begin{enumerate}
 \item $
 N\big( (a_{s-1}, t) \big)= \varphi_{l+1}[Y_{l+1}]$, which is
 $$\quad\qquad\quad\{(a_{s},1),  \dotsc, (a_{s}, \vert  w_l\vert_{a_s}+ \vert Y_{l+1}\backslash Y_l\vert  ), (a_{s-2},  1),  \dotsc,  (a_{s-2}, \vert Y_{l+1}\cap Y_l\vert -\vert w_l\vert_{a_s} )   \},$$ for each $1\leq t\leq \vert X_{l+1}\backslash X_l\vert$ ; and
 \item $N\big( (a_{s}, t') \big)= \varphi_{l+1}[X_{l+1}]$, which is
 $$\hspace{-43mm}\{(a_{s-1},1), (a_{s-1},2), \dotsc, (a_{s-1}, \vert X_{l+1}\vert)\}$$
  for each $1\leq t'\leq \vert Y_{l+1}\backslash Y_l\vert$.
  \end{enumerate}

Since $X_l\times Y_l\subseteq E(G_l)$, due to the isomorphism $\varphi_l$, it follows that
every $a_{s-1}$ in $w_l$ comes after the $(\vert Y_l\vert-\vert w_l\vert_{a_s})$-th $a_{s-2}$ in $w_l$ and every $a_s$ in $w_l$ comes after the $\vert X_l\vert$-th $a_{s-1}$ in $w_l$. Using these two facts, it can be argued as in Case~1 that properties $(1)$ and $(2)$ above hold.

Finally, $w_{l+1}$ and $\varphi_{l+1}$ easily satisfy the other required properties.\qedhere
\end{subcase}
	\end{case}
\end{proof}

\section{Hierarchy of Parikh Word Representable Graphs}

Suppose $\mathcal{PWG}_s$ denote the class of bipartite graphs Parikh word representable by some word over $\Sigma_s$. By Theorems~\ref{1207b} and \ref{0309c},
$\bigcup_{s=1}^\infty \mathcal{PWG}_s$ is the class of bipartite permutation graphs. Clearly, $\mathcal{PWG}_s \subseteq \mathcal{PWG}_{s+1}$ for all $s\geq 1$. This hierarchy of Parikh word representable graphs is strictly proper and to see this, we are studying upper bounds on the diameters of connected Parikh graphs in terms of where they are among the hierarchy.

\begin{lemma}\label{2507a}
Suppose $s\geq 2$, $w\in \Sigma_s^*$, and $\core_{a_{1,s}}(w)=w$. Then $\mathcal{G}(w)$ is connected and its diameter is at most $s+1$.
\end{lemma}

\begin{proof}
We argue by induction on $ s $. For $ s= 2 $, $ w = a_1xa_2 $ for some $ x \in \Sigma_2^* $ as $\core_{a_1a_2}(w)=w$. Since every vertex $ (a_2, k) $ for $ 1 \leq k \leq \vert w \vert_{a_2} $ is adjacent to the vertex $ (a_1, 1) $ and every vertex $ (a_1, l) $ for $ 1 \leq l \leq \vert w \vert_{a_1} $ is adjacent to the vertex $ (a_2, \vert w \vert_{a_2}) $, it follows that $\mathcal{G}(w)$ is connected and its diameter is at most three.
Hence, the base step holds.

For the induction step, suppose $ w $ is a word over $ \Sigma_{s+1} $ and is an $\core_{a_{1,s+1}}(w)=w$. Let $ w' = \pi_{a_1, a_2, \dotsc, a_s}(w) $. Note that $ w' $ is an $a_{1,s}$-core word and thus by the  induction hypothesis, $ \mathcal{G}(w') $ is connected and its diameter is at most $ s + 1 $. 
Futhermore, $\mathcal{G}(w')$ is exactly the subgraph of $\mathcal{G}(w)$ induced by vertices of the form $(a_i, k)$ for $1\leq i\leq s$ and $1\leq k\leq \vert w\vert_{a_i}$.
Since every vertex $ (a_{s+1}, l) $ for $ 1 \leq l \leq \vert w \vert_{a_{s+1}} $ in $\mathcal{G}(w)$ is adjacent to the vertex $ (a_s, 1) $ in $\mathcal{G}(w)$ as $\core_{a_{1,s+1}}(w)=w$, it follows that $ \mathcal{G}(w) $ is connected. 

Now, we show that the diameter of $ \mathcal{G}(w) $ is at most $ s + 2 $. Suppose $ x=(a_i, k) $ and $ y=(a_j, l) $ are two distinct vertices of $ \mathcal{G}(w) $, where $ 1 \leq i\leq j \leq s+1, 1 \leq k \leq \vert w \vert_{a_i} $, and $ 1 \leq l \leq \vert w \vert_{a_j} $.
If $ 1 \leq i\leq  j \leq s $, then  $x$ and $y$ are also vertices of $ \mathcal{G}(w')$ and thus $d(x,y)\leq s + 1$. Else if $ i = j = s + 1 $, then both $x$ and $y$ are  adjacent to the vertex $ (a_s, 1)$ and thus in fact, $d(x,y)=2$. Otherwise, suppose  $1\leq i\leq s$ and $j=s+1$. Then $y$ is adjacent to the vertex $y'=(a_s, 1)$.  Therefore,
$d(x,y)\leq d(x,y')+d(y',y)\leq (s+1)
+1 =s+2$. 
\end{proof}

\begin{theorem}\label{2407a}
	Suppose $s\geq 2$ and $w\in \Sigma_s^*$ with $\core_{a_{1,s}}(w)\neq \lambda$. Assume $\mathcal{G}(w)$ is connected. Then the diameter of $G$ is at most $s+3$ when $s\geq 3$ and at most three when $s=2$.
\end{theorem}

\begin{proof}
If $s=2$,  since $\mathcal{G}(w)$ is connected, the first (respectively last) letter of $w$ must be $a_1$ (respectively $a_2$)
and thus $\core_{a_1a_2}(w)=w$. Hence, the diameter of $\mathcal{G}(w)$ is at most three by Lemma~\ref{2507a}.

Suppose $s\geq3$. Note that $w\equiv_1 w_1 
\core_{a_{1,s}}(w)w_2$ for some $w_1\in \{a_2, a_3, \dotsc, a_s\}^*$ and
$w_2\in \{a_1, a_2, \dotsc, a_{s-1}\}^*$. Since words that are $1$-equivalent to each other have the same Parikh graph, we may assume $w= w_1
\core_{a_{1,s}}(w)w_2$.
Consider two vertices $x$ and $y$ of $\mathcal{G}(w)$. 
First, we assume one of them, say $x$, corresponds to a letter $a_i$ in $w$ that occurs in $\core_{a_{1,s}}(w)$. We will show that $d(x,y)\leq s+1$.

\setcounter{case}{0}
\begin{case}
$y$ also corresponds to a letter in $w$ that occurs in $\core_{a_{1,s}}(w)$.

Since $\core_{a_{1,s}}(\core_{a_{1,s}}(w))=	
\core_{a_{1,s}}(w)$, by Lemma~\ref{2507a}, the diameter of \linebreak $\mathcal{G}(\core_{a_{1,s}}(w))$ is at most $s+1$. It follows that $d(x,y)\leq s+1$.	
\end{case}

\begin{case}
$y$ corresponds to a letter $a_j$ in $w$ that occurs in $w_1$.	
	
\begin{subcase}
$2\leq j \leq s-1$.

We may assume $i\leq j$ as the other case is similar. By definition, $y$ is adjacent to the vertex $y'$ corresponding to the letter $a_{j+1}$ in $w$ that occurs last in $\core_{a_{1,s}}(w)$. Note that there exists an occurrence of the word $a_ia_{i+1}\dotsm a_{j+1}$ as a subword of $w$ for which $x$ corresponds to the letter $a_i$ in this occurrence and  $y'$ corresponds to the letter $a_{j+1}$ in this occurrence. Hence, $d(x,y')=j+1-i$. Therefore,
$d(x,y)\leq d(x,y')+d(y',y)=
j+1-i+1\leq (s-1)+1-1+1=s$.
\end{subcase}

\begin{subcase}
$j=s$	

Since $\mathcal{G}(w)$ is connected,  $y$ must be adjacent to a vertex $y'$ corresponding to a letter $a_{s-1}$ in $w$ that occurs in $w_1$. Hence,
$d(x,y)\leq d(x,y')+d(y',y)\leq s+1$,
where $d(x,y')\leq s$ by Case~2.1.
\end{subcase}
\end{case}

\begin{case}
	$y$ corresponds to a letter in $w$ that occurs in $w_2$.	
	
This is similar to Case 2.	
\end{case}

Finally, if neither  $x$ nor $y$ corresponds to a letter in $w$ that occurs in \linebreak
$\core_{a_{1,s}}(w)$, then we can choose a vertex $y'$ corresponding to a letter in $w$ that occurs in $\core_{a_{1,s}}(w)$ such that $d(y',y)=1 \text { or } 2$. 
Therefore, $d(x,y)\leq d(x,y')+d(y',y)\leq (s+1)+2=s+3$, where
$d(x,y')\leq s+1$ by what we have shown above.
	\end{proof}

The next example shows that the upper bound  in Theorem~\ref{2407a} is realizable.

\begin{example}
	Over $\{a<b<c<d<e\}$, the Parikh graphs 
	$\mathcal{G}(abab)$, $\mathcal{G}( bcabcab)	$, $\mathcal{G}( cdabcdab )	$, and $\mathcal{G}( deabcdeab)$ are path graphs of length $3$, $6$, $7$, and $8$ respectively.	
\end{example}

Surprisingly, we will see that if the requirement $\core_{a_{1,s}}(w)\neq \lambda$ is taken away, then the upper bound on the diameter of the connected Parikh graph $\mathcal{G}(w)$ increases significantly. 

\begin{lemma}\label{2809a}
Suppose $s\geq 3$ and $w\in \Sigma_s^*$ with $\supp(w)=\Sigma_s$. If $\mathcal{G}(w)$ is connected, then $\vert w\vert_{a_{i}a_{i+1}a_{i+2}}\neq 0$ for
every $1\leq i\leq s-2$.	
\end{lemma}

\begin{proof}
Fix an integer 	$1\leq i\leq s-2$. 
For any vertex $x$ in $\mathcal{G}(w)$ corresponding to a letter $a_i$ in $w$ and any vertex $y$ in $\mathcal{G}(w)$ corresponding to a letter $a_{i+2}$ in $w$, there exists a path in $\mathcal{G}(w)$ between $x$ and $y$ because $\mathcal{G}(w)$ is connected. Choose a path with such two endpoints having the smallest possible length. It is easy to see that this path must have length two where the intermediate vertex corresponding to a letter $a_{i+1}$ in $w$. By the definition of Parikh graph, the corresponding letters $a_i, a_{i+1}, a_{i+2}$ appear in $w$ from left to right. Hence, 
$\vert w\vert_{a_{i}a_{i+1}a_{i+2}}\neq 0$.	
\end{proof}

\begin{theorem}\label{0309a}
	Suppose $s\geq 2$ and $w\in \Sigma_s^*$. Assume $\mathcal{G}(w)$ is connected. Then the diameter of $G$ is at most $3s-3$.	
\end{theorem}

\begin{proof}
We argue by induction on $s$. The base step $s=2$ again holds as in the proof of Theorem~\ref{2407a}.
For the induction step, suppose $w\in \Sigma_{s+1}^*$
and $\mathcal{G}(w)$ is connected.
We may assume $\supp(w)=\Sigma_{s+1}$ for otherwise $\vert \supp(w)\vert \leq s$ and we are done by the induction hypothesis and some relabelling of the letters as $\supp(w)$ must consist of consecutive letters of $\Sigma$ due to connectivity. 
Let $w=w'w''w'''$ where the first $a_{s-1}$ of $w$ is the first letter of $w''$ and  the last $a_{s+1}$ of $w$ is the last letter of $w''$.
Note that $w'$, $w''$, and $w'''$ are well-defined and $\vert w''\vert_{a_s}\neq 0$ because
$\vert w\vert_{a_{s-1}a_{s}a_{s+1}}\neq 0$ by Lemma~\ref{2809a}.
Let $u= \pi_{a_{s}, a_{s+1}}(w'w'')$ and $v= \pi_{a_1, a_2, \dotsc, a_{s-2}}(w')\pi_{a_1, a_2, \dotsc, a_s}(w''w''')$. 
Let $\widehat{\mathcal{G}(v)}$ 
denote the corresponding subgraph 
of $\mathcal{G}(w)$ induced by the subword $v$, which is isomorphic to $\mathcal{G}(v)$. 

Due to connectivity of $\mathcal{G}(w)$, the first letter of $u$ must be $a_s$. Clearly the last letter of $u$ is $a_{s+1}$ because the last letter of $w''$ is $a_{s+1}$. Hence, $\mathcal{G}(u)$ is connected with diameter at most three.  
Since $\mathcal{G}(w)$ is connected and the first $a_{s-1}$ of $w$ is the first letter of $w''$,
it follows that
every vertex in $\widehat{\mathcal{G}(v)}$ is connected within $\widehat{\mathcal{G}(v)}$ to some vertex corresponding to a letter $a_s$ of $w$ occurring in $v$, which in turn is adjacent to the vertex $(a_{s-1},1)$ in $\widehat{\mathcal{G}(v)}$. It follows that $\widehat{\mathcal{G}(v)}$ and thus
$\mathcal{G}(v)$ is connected as well.
Meanwhile, note that $\mathcal{G}(u)$ is identical to the subgraph of $\mathcal{G}(w)$ induced by the subword $u$.

By the induction hypothesis, the diameter of $\mathcal{G}(v)$ is at most $3s-3$. Note that
$V(\mathcal{G}(u))\cup V(\widehat{\mathcal{G}(v)})=V(\mathcal{G}(w))$
and
 $V(\mathcal{G}(u))\cap V(\widehat{\mathcal{G}(v)})$ is nonempty because it contains
 exactly vertices of $\mathcal{G}(w)$ corresponding to the letters $a_s$'s appearing in $w''$.
   Therefore,
the diameter of $\mathcal{G}(w)$ is at most $(3s-3)+3=3s$.
\end{proof}

\begin{theorem}\label{0309e}
For every integer $s\geq 2$, the longest path graph that is Parikh word representable over $\Sigma_s$ has length $3s-3$.	
\end{theorem}

\begin{proof}
By Theorem~\ref{0309a}, it suffices to argue by induction that for every integer $s\geq 2$, there exists a word $w$ over $\Sigma_s$ such that $\mathcal{G}(w)$ is a path graph of length $3s-3$. 
For the basis step $s=2$, we can take $w=a_1b_1a_1b_1$. For the induction step, suppose a word $w$ with $\supp(w)=\Sigma_s$ such that $\mathcal{G}(w)$ is a path graph of length $3s-3$ and that the first $a_s$ of $w$ corresponds to one of the endpoint of $\mathcal{G}(w)$ has been constructed.
Consider the word $w'$ obtained from $w$ by prefixing $w$ with $a_sa_{s+1}$ and inserting $a_{s+1}$ immediately after the first $a_s$ in $w$. Then it can be verified that $\mathcal{G}(w')$ is a path graph of length $3s$ such that the first $a_s$ of $w'$ corresponds to one of the endpoint of $\mathcal{G}(w')$. 
\end{proof}

\begin{example}
 Over $\{a<b<c<d<e\}$, the Parikh graphs 
 $\mathcal{G}(abab)$, $\mathcal{G}( \mathbf{bc}ab \mathbf{c}ab   )	$, $\mathcal{G}( \mathbf{cd}bc \mathbf{d}abcab )	$, and $\mathcal{G}( \mathbf{de}cd \mathbf{e}bcdabcab )$ are path graphs of length $3$, $6$, $9$, and $12$ respectively.	
\end{example}

\begin{corollary}\label{1907a}
$\mathcal{PWG}_s\subsetneq \mathcal{PWG}_{s+1}$
	for all integer $s\geq 1$. 	
\end{corollary}

\begin{proof}
The  inclusion is strict by Theorem~\ref{0309e}.	
	\end{proof}

\section{Parikh Binary or Ternary  Word Representability}

 The following characterization of Parikh binary word representability that has not appeared explicitly is more in line with the characterization of bipartite permutation graphs in terms of  neighborhoods \cite[Theorem~1]{spinrad1987bipartite}.

\begin{theorem}\label{0207a}
	Suppose $G=(X,Y,E)$ is a connected bipartite graphs. Then $G$ is Parikh binary word representable if and only if 
	$N(x)\subseteq N(x')$ or $N(x')\subseteq N(x)$ whenever $x,x'\in X$.
\end{theorem}

\begin{proof}
	Suppose $ G $ is Parikh binary word representable. Then without loss of generality, $ G= \mathcal{G}(w) $ for some word $ w$ over $\Sigma=\{a<b\}$. By definition, $X$ either consists of vertices corresponding to the letters in $w$ which are $a$ or 
	consists of vertices corresponding to the letters in $w$
	which are $b$. In either case, it is clear that $N(x)\subseteq N(x')$ or $N(x')\subseteq N(x)$ whenever $x,x'\in X$.

	Conversely, suppose	$N(x)\subseteq N(x')$ or $N(x')\subseteq N(x)$ whenever $x,x'\in X$. Then due to finiteness of the neighborhoods,
	there is some listing $x_1, x_2, \dotsc, x_{\vert X\vert}$ of the elements of $X$ such that $N(x_1)\supseteq N(x_2)\supseteq \dotsb \supseteq N(x_{\vert X\vert})$. Due to connectivity, it follows that $N(x_{\vert X\vert})\neq \emptyset$
	and $N(x_1)=Y$.	
	Let $y_1, y_2, \dotsc, y_{\vert Y\vert}$ be some listing  of the elements of $Y$ such that for every $1\leq i \leq \vert X\vert$, $N(x_i)= 
\{y_{\vert Y\vert-\vert N(x_i)\vert+1}, y_{l+1}, \dotsc, y_{\vert Y\vert}\}$.	
	Let
	$$w=  ab^{\vert N(x_1)\backslash N(x_2)\vert}ab^{\vert N(x_2)\backslash N(x_3)\vert}\dotsm ab^{\vert N(x_{\vert X\vert-1})\backslash N(x_{\vert X\vert})\vert}ab^{\vert N(x_{\vert X\vert}) \vert}.$$
Consider the mapping $\varphi \colon X\cup Y \rightarrow V(\mathcal{G}(w))$ defined by $\varphi(x_i)= (a,i)$ for $1\leq i\leq \vert X\vert$
and $\varphi(y_i)= (b,i)$ for $1\leq i\leq \vert Y\vert$. We claim that $\varphi$ is an isomorphism between $G$ and $\mathcal{G}(w)$. Suppose $1\leq i\leq \vert X\vert$. 
Since $N(x_i)= 
\{y_{\vert Y\vert-\vert N(x_i)\vert+1}, y_{l+1}, \dotsc, y_{\vert Y\vert}\}$, it suffices to show that $\varphi(x_i)$ is adjacent to
$\varphi(y_j)$ in $\mathcal{G}(w)$ if and only if
$\vert Y\vert-\vert N(x_i)\vert+1\leq j\leq \vert Y\vert$,  provided $N(x_i)$ is nonempty. From the definition of $w$, there are exactly 
$$\vert N(x_1)\backslash N(x_2)\vert
+\dotsb+ \vert N(x_{i-1})\backslash N(x_i)\vert=
   \vert N(x_1)\vert-\vert N(x_i)\vert= \vert Y\vert -\vert N(x_i)\vert$$ many $b$'s before the $i$-th $a$. Hence, by definition, the vertex $(a,i)$ in $\mathcal{G}(w)$ is adjacent to the vertex $(b,j)$ for every $\vert Y\vert-\vert N(x_i)\vert+1\leq j\leq \vert Y\vert$ as required.
\end{proof}

Using Theorem~\ref{0207a}, we now give an alternative and direct proof (without using induction) of the following clever known result. It is a well-established fact that a graph is chordal if and only if it is an intersection graph of subtrees of a tree (see Theorem 1.2.3 of \cite{brandstadt1999graph}).  

\begin{theorem}\cite{bera2016structural}\label{1807c}
	A connected bipartite graph $G=(X,Y,E)$ is Parikh binary word representable if and only if 
	\begin{enumerate}
		\item  $G$ is a $(6,2)$ chordal graph; and
		\item there are two adjacent vertices whose degree sum is the same as the number of vertices of $G$.	
	\end{enumerate}	
\end{theorem}

\begin{proof}
	Suppose $G$ is Parikh binary word representable. Without loss of generality, we may assume $G=\mathcal{G}(w)$ for some word $w$ over $\{a<b\}$.
	Since $\mathcal{G}(w)$ is connected, the first letter of $w$ must be $a$ and the last letter of $w$ must be $b$. Furthermore, the vertices of $\mathcal{G}(w)$ corresponding 
	to these two letters are adjacent and their degree sum is the same as the number of vertices of $\mathcal{G}(w)$.
	Now, suppose $C$ is a cycle of length six in $\mathcal{G}(w)$. Exactly three of the vertices in $C$ must belong to $X$, say $x$, $x'$ and $x''$. By Theorem~\ref{0207a},  we may assume
	$N(x)\subseteq N(x')\subseteq N(x'')$. Hence, the two vertices in $C$ adjacent to $x$ also belong to $N(x')$ and $N(x'')$. Then it can be inferred that $C$ has at least two chords.
	
	Conversely, let $x\in X$ and $y\in Y$ be two adjacent vertices  in $\mathcal{G}(w)$ whose degree sum
	is the same as the number of vertices of $G$.
	This implies that $x$ is adjacent to every vertex in $Y$ and $y$ is adjacent to every vertex in $X$. Assume $G$ is not Parikh binary word representable. Then by Theorem~\ref{0207a}, there exist $x',x''\in X$ such that $N(x')\nsubseteq N(x'')$ and 
	$N(x'')\nsubseteq N(x')$. Let $y'\in N(x')\backslash N(x'')$ and
	$y''\in N(x'')\backslash N(x')$. Consider the cycle $C$ with vertices $x,y',x',y,x'',y'',x$ in the given order. Then the edge $(x,y)$ is the only chord of $C$, a contradiction.
\end{proof}

Now, we further extend Theorem~\ref{0207a} to the ternary alphabet.

\begin{theorem}\label{0209a}
	Suppose $G=(X,Y,E)$ is a connected bipartite graph. Then $G$ is Parikh ternary word representable if and only if there exists an ordering on one of the parts, say $Y$, such that for every $x\in X$, $N(x)$ is either an initial segment or an end segment of $Y$.	
\end{theorem}

\begin{proof}
Suppose $ G $ is Parikh ternary word representable. Without loss of generality, suppose $ G = \mathcal{G}(w) $ for some word $ w $ over $\{a < b < c\} $ and $ X $ consists of vertices corresponding to the letters in $ w $ which are $ a $ or $ c $. Consider the ordering $ (b, 1), (b, 2), \dots, (b, \vert w \vert_b) $ on $ Y $. It is clear that for every $ x \in X $ that corresponds to a letter in $ w $ which is $ c $ (respectively $ a $), $ N(x) $ is an  initial segment (respectively end segment) of $ Y $ by the definition of Parikh graph.

Conversely, suppose there exists an ordering  $y_1, y_2, \dotsc, y_{\vert Y\vert}$ say on $ Y $ such that for every $ x \in X $, $ N(x) $ is either an initial segment or an end segment of $ Y $. We may suppose both types of segments appear for otherwise $ G $ is Parikh binary word representable by Theorem~\ref{0207a}. Then there is some ordering $ x_1, x_2, \dots, x_k, x'_1, x'_2, \dots, x'_{k'} $ on $ X $ for some $ 1 \leq k,k' < \vert X \vert $ with $ k + k' = \vert X \vert $
such that $ N(x_i) $ is an initial segment of $ Y $ for every $ 1 \leq i \leq k $ with $\emptyset\neq N(x_1) \subseteq N(x_2) \subseteq \dotsb \subseteq N(x_k) $ and  $ N(x'_j) $ is an end segment of $ Y $ for every $ 1 \leq j \leq k' $ with $ \emptyset \neq N(x'_1) \subseteq N(x'_2) \subseteq \dotsb \subseteq N(x'_{k'})$. However,  $N(x_k)$ and $N(x'_{k'})$ need not be $Y$.

Consider the ordered alphabet $\Sigma=\{a<b<c\}$. Let $$v = b^{\vert N(x_1) \vert}cb^{\vert N(x_2) \backslash N(x_1) \vert}c \dotsb b^{\vert N(x_k) \backslash N(x_{k-1}) \vert}cb^{\vert Y \backslash N(x_k) \vert} $$ and let  $$ v' = b^{\vert Y \backslash N(x'_{k'}) \vert}ab^{\vert N(x'_{k'}) \backslash N(x'_{k'-1}) \vert} \dotsb ab^{\vert N(x'_2) \backslash N(x'_1) \vert}ab^{\vert N(x'_1) \vert} .$$ Take any word $ w \in \Sigma^* $ such that $ \pi_{a,b}(w) = v' $ and $ \pi_{b,c}(w) = v$. Consider the mapping $\varphi \colon X\cup Y \rightarrow V(\mathcal{G}(w))$ defined by 
\begin{align*}
 \varphi(x'_i) ={}& (a, i) \text{ for } 1 \leq i \leq k',\\
\varphi(x_i) ={}& (c, i) \text{ for } 1 \leq i \leq k,\\
 \varphi(y_i) ={}& (b, i) \text{ for } 1 \leq i \leq \vert Y \vert. 
\end{align*}
Then it can be verif\mbox{}ied as in the proof of Theorem~\ref{0207a} that $\varphi$ is an isomorphism between $G$ and $\mathcal{G}(w)$.
\end{proof}

The first graph theoretic property investigated in the introductory paper on bipartite permutation graphs is the existence of a Hamiltonian cycle.

\begin{theorem}\cite{spinrad1987bipartite}\label{2607a}
	Suppose $G=(X,Y,E)$ is a connected bipartite permutation graph with $\vert X\vert=\vert Y\vert=k\geq 2$ and 
$(<_X,<_Y)$ is	the corresponding strong ordering, where	
	$<_x:x_1,x_2, \dotsc, x_k$ and $<_Y:y_1,y_2, \dotsc, y_k$.  Then $G$ has a Hamiltonian cycle if and only if $x_i, y_i, x_{i+1}, y_{i+1}, x_i$ is a cycle of length four for each $1\leq i\leq k-1$.	
\end{theorem}	

Suppose $\Sigma=\{a<b\}$ and $w\in \Sigma^*$ with $\vert w\vert_a=\vert w\vert_b$. Using Theorem~\ref{2607a}, it can be deduced that
$\mathcal{G}(w)$ has a Hamiltonian cycle if and only if every proper prefix of $w$ has more number of $a$'s then $b$'s \cite[Theorem~7]{bera2016structural}.
Now, we give the corresponding result for the ternary alphabet.

\begin{theorem}
	Suppose $\Sigma=\{a<b<c\}$ and $w\in \Sigma^*$ with $\vert w\vert_a+\vert w\vert_c=\vert w\vert_b$. 
	Assume $\mathcal{G}(w)$ is connected.
	Then $\mathcal{G}(w)$ has a Hamiltonian cycle if and only if 
	every proper prefix of $w$ has (strictly) more number of $b$'s than $c$'s
	and every proper suffix of $w$ has (strictly) more number of $b$'s than $a$'s.
\end{theorem}

\begin{proof}
	Let $\mathcal{G}(w)=(X,Y,E)$, where $X$ is the part corresponding to the letters $a$ and $c$.
	By Lemma~\ref{0409b}, we know that 
	$(<_X, <_Y)$ is a strong ordering on the vertices of $\mathcal{G}(w)$, where
	$<_X:(c,1)<(c,2)< \dotsb <(c,\vert w\vert_c)<
	(a,1)<(a,2)< \dotsb <(a,\vert w\vert_a)$
	and $<_Y:(b,1)<(b,2)< \dotsb <(b,\vert w\vert_b)$.
	
	Suppose $\mathcal{G}(w)$ has a Hamiltonian cycle.
	By Theorem~\ref{2607a}, $(c,i)$ is adjacent to $(b,i+1)$ for all $1\leq i \leq \vert w\vert_c$. Hence, 	this means that the $i$-th $c$ in $w$ comes after the $(i+1)$-th $b$ in $w$. It follows that	
	 every proper prefix of $w$ has more number of $b$'s than $c$'s. Similarly, by Theorem~\ref{2607a}, $(a,\vert w\vert_a-i+1)$ is adjacent to $(b,\vert w\vert_b -i)$ for all $1\leq i \leq \vert w\vert_a$.  This means that the $i$-th $c$ in $w$ counting from the rear comes before the $(i+1)$-th $b$ in $w$ counting from the rear. It follows that every proper suffix of $w$ has more number of $b$'s than $a$'s. 
	
	Conversely, suppose every proper prefix of $w$ has more number of $b$'s than $c$'s
		and every proper suffix of $w$ has more number of $b$'s than $a$'s. It follows that
	$(c,i)$ is adjacent to $(b,i+1)$ for all $1\leq i \leq \vert w\vert_c$ and $(a,\vert w\vert_a-i+1)$ is adjacent to $(b,\vert w\vert_b -i)$ for all $1\leq i \leq \vert w\vert_a$. Then Theorem~\ref{2607a} can be applied to conclude that $\mathcal{G}(w)$ has a Hamiltonian cycle. (In particular, it can be verified that $(c, \vert w\vert_c),  	
(b, \vert w\vert_{c}), (a,1),
(b, \vert w\vert_c+1), (c, \vert w\vert_c)$ is a cycle of length four in $\mathcal{G}(w)$.)		 
\end{proof}

\section{conclusion}

In this work, we have continued on the pioneering study on Parikh word representable graphs in \cite{bera2016structural} and expanded it soundly beyond the binary alphabet. In fact, our study mostly focuses on the general theory for any ordered alphabets.

We have shown that the class of Parikh word representable graphs is equivalent to the class of bipartite permutation graphs that is a special class of intersection graphs. Using Scheinerman's  characterization \cite{scheinerman1985characterizing}  of 
classes of finite graphs that can be described as classes of intersection graphs,  it is possible to show that the class of Parikh word representable graphs is a 
class of intersection graphs directly, thus it is not surprising that the former class coincides with the class of bipartite permutation graphs.

Finally, although the definition of Parikh graph is originally motivated
by Parikh matrices, no comprehensive study has been made that relates Parikh graphs to Parikh matrices. Therefore, as a possible future direction, one can attempt to address this to some extent.

\section*{Acknowledgments}
The first and last authors gratefully acknowledge the support for this research by Research University Grant No.~1001/PMATHS/8011019 of Universiti Sains Malaysia.


\begin{thebibliography}{10}
	
	\bibitem{bera2016structural}
	S.~Bera and K.~Mahalingam.
	\newblock Structural properties of word representable graphs.
	\newblock {\em Math. Comput. Sci.}, 10(2):209--222, Jun 2016.
	
	\bibitem{brandstadt1999graph}
	A.~Brandstadt, J.~P. Spinrad, et~al.
	\newblock {\em Graph classes: a survey}, volume~3.
	\newblock Siam, 1999.
	
	\bibitem{chalopin2009every}
	J.~Chalopin and D.~Gon{\c{c}}alves.
	\newblock Every planar graph is the intersection graph of segments in the
	plane.
	\newblock In {\em Proceedings of the forty-first annual ACM symposium on Theory
		of computing}, pages 631--638. ACM, 2009.
	
	\bibitem{hell2004interval}
	P.~Hell and J.~Huang.
	\newblock Interval bigraphs and circular arc graphs.
	\newblock {\em Journal of Graph Theory}, 46, 2004.
	
	\bibitem{kitaev2015words}
	S.~Kitaev and V.~Lozin.
	\newblock {\em Words and graphs}.
	\newblock Monographs in Theoretical Computer Science. 
	Springer, Cham, 2015.
	
	\bibitem{mateescu2001sharpening}
	A.~Mateescu, A.~Salomaa, K.~Salomaa, and S.~Yu.
	\newblock A sharpening of the {P}arikh mapping.
	\newblock {\em Theor. Inform. Appl.}, 35(6):551--564, 2001.
	
	\bibitem{mckee1999topics}
	T.~A. McKee and F.~R. McMorris.
	\newblock {\em Topics in intersection graph theory}.
	\newblock SIAM Monographs on Discrete Mathematics and Applications. Society for
	Industrial and Applied Mathematics (SIAM), Philadelphia, 1999.
	
	\bibitem{parikh1966context}
	R.~J. Parikh.
	\newblock On context-free languages.
	\newblock {\em J. Assoc. Comput. Mach.}, 13:570--581, 1966.
	
	\bibitem{poovanandran2018elementary}
	G.~Poovanandran and W.~C. Teh.
	\newblock Elementary matrix equivalence and core transformation graphs for
	{P}arikh matrices.
	\newblock {\em Discrete Appl. Math.}, 251:276 -- 289, 2018.
	
	\bibitem{scheinerman1985characterizing}
	E.~R. Scheinerman.
	\newblock Characterizing intersection classes of graphs.
	\newblock {\em Discrete Math.}, 55(2):185--193, 1985.
	
	\bibitem{spinrad1987bipartite}
	J.~Spinrad, A.~Brandst{\"a}dt, and L.~Stewart.
	\newblock Bipartite permutation graphs.
	\newblock {\em Discrete Appl. Math.}, 18(3):279--292, 1987.
	
	\bibitem{teh2015coreb}
	W.~C. Teh.
	\newblock On core words and the {P}arikh matrix mapping.
	\newblock {\em Internat. J. Found. Comput. Sci.}, 26(1):123--142, 2015.
	
	\bibitem{teh2016conjecture}
	W.~C. Teh and A.~Atanasiu.
	\newblock On a conjecture about {P}arikh matrices.
	\newblock {\em Theoret. Comput. Sci.}, 628:30--39, 2016.
	
	\bibitem{teh2018strongly}
	W.~C. Teh, A.~Atanasiu, and G.~Poovanandran.
	\newblock On strongly \textit{M}\!-unambiguous prints and
	\c{S}erb\v{a}nu\c{t}\v{a}'s conjecture for {P}arikh matrices.
	\newblock {\em Theoret. Comput. Sci.}, 719:86--93, 2017.
	
	\bibitem{teh2018order}
	W.~C. Teh, K.~Subramanian, and S.~Bera.
	\newblock Order of weak \textit{M}-relation and {P}arikh matrices.
	\newblock {\em Theoret. Comput. Sci.}, 743:83 -- 92, 2018.
	
\end{thebibliography}

\end{document}